\documentclass[11]{article}

\usepackage[utf8]{inputenc}

\usepackage[colorlinks=false]{hyperref}
\usepackage{url}
\usepackage{amsmath}
\usepackage[cal=dutchcal]{mathalfa}
\usepackage{standalone}
\usepackage{color}
\usepackage{tikz}
\usepackage{placeins}
\usepackage{amsfonts}
\usepackage{mathtools}
\usepackage{enumerate}
\usepackage{mathrsfs}
\usetikzlibrary{decorations.markings,arrows}
\usetikzlibrary{calc}
\usepackage{graphicx}
\usepackage{array}
\usepackage{float}
\usepackage{booktabs}
\newcommand{\ra}[1]{\renewcommand{\arraystretch}{#1}}
\usepackage{amsthm}
\hypersetup{ colorlinks, citecolor=green, filecolor=black, linkcolor=blue, urlcolor=blue }
\newtheorem{theorem}{Theorem}[section]
\newtheorem{corollary}{Corollary}[theorem]
\newtheorem{lemma}[theorem]{Lemma}

\newtheorem{definition}[theorem]{Definition}

\usepackage{graphicx}
\usepackage{indentfirst}
\usepackage{enumitem}
\usepackage{comment}
\usepackage{titling}
\usepackage{siunitx}
\usepackage{bm}

\definecolor{ballblue}{rgb}{0.13, 0.67, 0.8}

\author{Vasiliki Velona\thanks{Department of Mathematics, Universitat Politècnica de Catalunya, Barcelona, Spain. 
}~\thanks{Department of Economics, Universitat Pompeu Fabra, Barcelona, Spain.}~\thanks{Email: \href{mailto:vasiliki.velona@upf.edu}{vasiliki.velona@upf.edu}}}
\title{Encoding and avoiding 2-connected patterns in polygon dissections and outerplanar graphs}

\begin{document}

\maketitle
\begin{abstract}
Let $\Delta =\{ \delta _1,\delta _2,...,\delta _m \} $ be a finite set of 2-connected patterns, i.e. graphs up to vertex relabelling. We study the generating function $D_{\Delta }(z,u_1,u_2,...,u_m),$ which counts polygon dissections and marks subgraph copies of $\delta _i$ with the variable $u_i$. We prove that this is always algebraic, through an explicit combinatorial decomposition depending on $\Delta $. The decomposition also gives a defining system for $D_{\Delta }(z,\mathbf{0})$, which encodes polygon dissections that avoid these patterns as subgraphs. In this way, we are able to extract normal limit laws for the patterns when they are encoded, and perform asymptotic enumeration of the resulting classes when they are avoided. The results can be transfered to the case of labelled outerplanar graphs. We give examples and compute the relevant constants when the patterns are small cycles or dissections.
\end{abstract}

\section{Introduction}

The study of subgraph appearances in random graph models is a well established line of research, beginning with the classic \emph{Erd\H{o}s-R\'enyi} graph and results concerning the distribution of such appearances and threshold phenomena, as in~\cite{karonski1983number},\cite{rucinski1988small}. In parallel, attention was also drawn on models where, given some well-known graph class, an object is chosen uniformly at random from all the objects of size $n$; see for instance~\cite{kim2007small} and \cite{gao2008distribution} for regular graphs. In the last decades, techniques using a mixture of generating function theory and analytic tools have evolved significantly and are in the centre of such advances for various other graph classes. A number of graph statistics, such as number of components, edges, cut vertices, triangles, chromatic number and others, have been studied for standard graph classes, such as planar graphs, outerplanar, series-parallel, graphs of fixed genus, and minor-closed families; see for instance \cite{chapuy2011asymptotic}, \cite{bodirsky2007enumeration}, \cite{gimenez2013graph}, \cite{mcdiarmid2009random}. 

In~\cite{subcritical}, the authors present a normality result for the so-called \emph{subcritical} family of graphs, that contains standard graph classes such as trees, cacti graphs, outerplanar, and series-parallel graphs. In particular, all subgraph parameters in such a class follow a normal limit law, with linear mean and variance. However, no constructive way is given in it, in order to compute the corresponding constants for the mean and variance. One of the results of this work is an explicit way to do so in outerplanar graphs, for any set of 2-connected patterns, i.e. graphs up to vertex relabelling. As a case study, we examine 3 and 4-cycles, but the process by which these constants are obtained can be directly transferred to the case of any set of 2-connected parameters.

\begin{theorem}\label{222}
The number of appearances $X_n$ of 3-cycles and 4-cycles in polygon dissections and outerplanar graphs of size $n$ follows a normal limit law, as in~\ref{quasi}, where the mean and variance are asymptotically linear, i.e. $\mathbb{E}[X_n]=\mu n+\mathcal{O}(1)$ and $ \mathbb{V}\textsl{ar}\hspace{.03cm}[X_n]=\sigma ^2 n+\mathcal{O}(1)$. The constants $\mu $ and $\sigma ^2$ are the following, in their exact values for dissections and in approximation for outerplanar graphs:
\begin{table}[h!]\centering
\begin{tabular}{l|ll|ll}
Parameter  &  $\mu $ & $\sigma ^2 $ &  $\mu $ & $\sigma ^2 $\\
\midrule[.8pt]

3-cycles   &$\frac{1}{2}$ & ${\frac {-13+9\,\sqrt {2}}{-12+8\,\sqrt {2}}}\approx 0.39644$&  0.34793 &  0.40737  \\ 

4-cycles     & $ {\frac {-30+21\,\sqrt {2}}{-12+8\,\sqrt {2}}}

\approx 0.43933$ & $\,{\frac {-24216+17123\,\sqrt {2}}{-32 \left( -3+2\,\sqrt {2}
 \right) ^{2}}}
\approx 0.44710$ &  0.33705  &  0.36145

\end{tabular}
\end{table}
\end{theorem}

A necessary step for the analysis of outerplanar graphs is the analysis of polygon dissections, denoted by $\mathcal{D}$, with some fixed numbering on the vertices. For a finite set of 2-connected patterns $\Delta =\{ \delta _1,\delta _2,...,\delta _k \} $, we prove a combinatorial decomposition of $\mathcal{D}$ that allows the encoding of such patterns, depending on $\Delta $. In this way, we obtain defining systems for the multivariate generating function $D_{\Delta }(z,u_1,u_2,...,u_k)$, where the coefficient of $z^nu_1^{n_1}\cdots u_m^{n_m}$ counts the number of $\alpha \in\mathcal{D}$ that have $n$ vertices and $n_i$ subgraph occurrences of the pattern $\delta _i$. 

This task is of independent interest, as it is related to the enumeration problem of polygon dissections, a line of work that is quite old. Starting from the enumeration of polygon triangulations with Euler and Segner in the 18th century, a great amount of work has been devoted up until today to relevant problems. Usually, these problems put restrictions either on the number or the size of the partition's polygonal components, or even colour restrictions, recently; see for instance \cite{cayley}, \cite{read1978general}, \cite{birmajer2017colored}. However, the problem where a whole pattern is avoided as subgraph (i.e., cannot be recovered by applying edge and vertex deletions) seems to not have been studied at all, except for the case of triangle freeness in~\cite{birmajer2017colored}, where the problem the authors are dealing with does not concern subgraph restrictions, but restrictions on the type and colour of the partition's polygonal components. With results of this work, it is possible to handle subgraph restrictions of any set $\Delta $ and perform asymptotic enumeration of the resulting classes. We give such examples. In fact, we obtain the following results (corresponding to Corollary~\ref{cor} and Theorem~\ref{enumm}, respectively):

\begin{theorem}\label{alg}
The generating function $D_{\Delta }(z,\mathbf{u})$ is algebraic and the defining polynomial is computable. The generating function of polygon dissections that avoid all $\Delta $-patterns as subgraphs, $D_{\Delta }(z,\mathbf{0})$, is likewise algebraic.
\end{theorem}

\begin{theorem}\label{ex}
Let $\mathcal{D}, \mathcal{G}$ be the classes of dissections and outerplanar graphs avoiding a set of 2-connected patterns $\Delta =\{\delta _1 ,...,\delta _m\}$, respectively. Then, $\mathcal{D}$ and $\mathcal{G}$ have asymptotic growth of the form: \[\alpha _n\sim \frac{\alpha}{\Gamma (-\frac{1}{2})}\cdot n^{-3/2}\cdot r ^{-n} \quad \mathrm{and}\quad g_n\sim \frac{g}{\Gamma (-\frac{3}{2})}\cdot n^{-5/2} \cdot \rho ^{-n}\cdot n!,\] respectively, where both $\alpha ,g$ are computable constants. In Table~\ref{table:3}, there are approximations of $\alpha , g$ for various choices of $\Delta $. \end{theorem}

We also prove a multivariate central limit theorem for the number of appearances of 2-connected patterns in polygon dissections (corresponding to Theorem~\ref{random}): 
\begin{theorem}\label{1st}
Let $\Delta =\{\delta _1,...,\delta _m\}$ be a set of 2-connected patterns. Let $\Omega _n$ be the set of polygon dissections of size $n$ and $X_n:\Omega _n \rightarrow \mathbb{Z}_{\geq 0}^m$ be a vector of random variables $X_{\delta _1},...,X_{\delta _m}$ in $\Omega _n$, such that $X_{\delta _i}(\omega )$ is the number of $\delta _i$ patterns in $\omega \in\Omega _n$. Then, $\mathbf{X}_n$ satisfies a central limit theorem \[\frac{1}{\sqrt{n}}(\mathbf{X}_n-\mathbb{E}[\mathbf{X}_n])\xrightarrow[]{d} N(\mathbf{0},\mathbf{\Sigma })\] with 
\[\mathbb{E}[\mathbf{X}_n]=\boldsymbol{\mu }n+\mathcal{O}(1) \text{ and } \mathbb{C}\textsl{ov}\hspace{.03cm}[\mathbf{X}_n]=\mathbf{\Sigma}n+\mathcal{O}(1),\] where $\boldsymbol{\mu }$ and $\mathbf{\Sigma}$ are computable. 
\end{theorem}

There are some natural questions arising from this work. One is whether it is possible to extend the combinatorial construction that is proved for general parameters, with multiple cut vertices, and how. Also, one might wonder in which other combinatorial structures we can apply this reasoning, apart from outerplanar graphs. An example for the latter can be found in the dual class of polygon dissections, planted plane trees with outdegrees in $\mathbb{N}\setminus \{1\}$, denoted by $\mathcal{T}$. Consider as parameter in $T\in\mathcal{T}$ the number of subtrees $T'$ with $k$ leaves, such that $\deg _T(v)=\deg _{T'}(v)$ for each node $v$ that is inner in $T'$. Then, the equivalent parameter for polygon dissections is the number of $k$-cycles.
\newline\newline
\textbf{Plan of the paper.} In Section 2, we mention definitions and theorems that will be used. In Section 3, we prove a combinatorial decomposition of $\mathcal{D}$ depending on $\Delta $ and then Theorem~\ref{alg}. We also prove Theorem~\ref{1st}. In section 4, we give applications of the previous and prove Theorem~\ref{222} and Theorem~\ref{ex}. In the Appendix, Table~\ref{count} contains the initial terms of all the counting sequences appearing in Section 4.


\section{Preliminaries}

The framework we use is the \emph{symbolic method} and the corresponding analytic techniques, as they were presented in \cite{flajolet1999analytic}. 
\paragraph{Symbolic methods for counting.}

A \emph{combinatorial class} is a set $\mathcal{A}$ with a \emph{size function} $\mathcal{A}\rightarrow \mathbb{Z}_{\geq 0}$, such that the inverse image of any integer is a finite set, denoted $\mathcal{A}_n$. Each $\alpha\in\mathcal{A}_n$ comprises $n$ \emph{atoms} of size 1, and we denote by $\mathcal{Z}$ the \emph{atomic class} that contains exactly one object of size one. In this work, atoms always represent graph vertices. We call a class $\mathcal{A}$ \emph{labelled} if the atoms have labels and $\mathcal{A}$ is closed under atom relabelling. The \emph{ordinary generating function} $A(z)$ of $\mathcal{A}$, referred also as \emph{ogf}, is defined as $\sum _{n=0}^{\infty }|\mathcal{A}_n|z^n$. If $\mathcal{A}$ is labelled, we use the \emph{ exponential generating function} $A(z)$, referred also as \emph{egf}, that is defined as $\sum _{n=0}^{\infty }|\mathcal{A}_n|\frac{z^n}{n!}$. We then write $[z^n]A(z)=|\mathcal{A}_n|$ for ogfs and $[z^n]A(z)=\frac{|\mathcal{A}_n|}{n!}$ for egfs.\footnote{From now on, generating functions will be denoted by plane letters and combinatorial classes by calligraphic letters.}

In order to create functional equations for the generating functions of interest, we use the so-called \emph{admissible} combinatorial constructions from \cite{flajolet1999analytic}. The aim is to express a combinatorial class in terms of other ones, itself included, in an \textit{admissible way}. Then, there is a direct translation in terms of generating functions. From unlabelled classes and ogfs, we only need the elementary cases $\mathcal{A}=\mathcal{B}\cup\mathcal{C}\Rightarrow A(z)=B(z)+C(z)$ and $\mathcal{A}=\mathcal{B}\times\mathcal{C}\Rightarrow A(z)=B(z)\cdot C(z)$. In Table~\ref{table:1}, there are all the labelled constructions that are useful to this work, along with their translations to egfs. 

It is useful to consider \textit{parameters} on the objects of $\mathcal{A},$ i.e functions $\chi _i:\mathcal{A}\rightarrow \mathbb{Z}_{\geq 0}$ that quantify some structure of the objects. Let $\mathbf{j} $ be $(j_1,...,j _{m})$\footnote{This convention is followed  in an analogous way for all bold characters.}, where $j_i \in \mathbb{Z}_{\geq 0}$ and let us define $\mathcal{A} _{n,\mathbf{j}}$ as the set of elements $\alpha\in\mathcal{A}$ that have size $n$ and $\chi _{i}(\alpha )=j _{i}.$ Then we work with multivariate generating functions, ordinary $\sum _{n,j _i\geq 0}|\mathcal{A}_{n,\mathbf{j} }|z^{n}u_1^{j _1}u_1^{j _2}...u_m^{j  _{m}}$ for unlabelled classes and exponential $\sum _{n,j _i\geq 0}|\mathcal{A} _{n,\mathbf{j} }|\frac{z^{n}}{n!}u_1^{j _1}u_2^{j _2}...u_m^{j  _{m}}$ for labelled ones. All the mentioned translations are also valid for multivariate generating functions, if the parameters are \emph{compatible}, i.e. $\chi (\alpha ')$ is the same for all order-preserving relabelings $\alpha '$ of $\alpha \in \mathcal{A}$, and  \emph{additive}, i.e. $\chi (\alpha )=\sum _i \chi (\beta _i)$ when $\alpha \in \mathcal{A}$ is composed of smaller elements $\beta _i\in \mathcal{B}$.

A generating function $y(z,\mathbf{u})$ is called \emph{algebraic} if it satisfies a polynomial equation $P(z,y,\mathbf{u})=0$.

\begin{table}[h]\centering
\begin{tabular}{@{ }cc|c||c@{ }}

\textit{Labelled product:} & $\mathcal{B}\star \mathcal{C}$ & $B(z)\cdot C(z)$ &   \\

\textit{Set:} & $\mathrm{S\scalebox{.9}{ET}}(\mathcal{B})$ & $\exp (B(z))$ & \\
 \textit{Substitution:} & $\mathcal{B}\circ \mathcal{C}$ & $ B(C(z)) $ \\
\end{tabular}
\begin{tabular}{@{ }cc|c@{ }}

\textit{Pointing:} & $\mathcal{B}^{\bullet}$ & $ z\partial _z B(z)
 $\vspace{-.3cm} \\

 & &  \\

\textit{Deriving:}  & $\mathcal{B}^{\circ}$ & $ \partial _z B(z)
 $ \\

\end{tabular}
\caption{Some labelled constructions and their translations to exponential generating functions.}\label{table:1}
\end{table}

\paragraph{Graph theoretic preliminaries.}

We now mention some basic graph theoretic language and refer to~\cite{diestel2010graph} for a rigorous exposition.

A \emph{graph} $G(V,E)$ is defined by the vertex set $V$ and the edge set $E$ that is a set of 2-element subsets of $V$. If the elements of $E$ are ordered pairs of vertices, the graph is called \emph{directed}. A graph $G_1$ is a \emph{subgraph} of $G_2$ if it can be obtained by $G_2$ with edge and vertex deletions. In this work, a \emph{pattern} is the equivalence class of a graph, up to vertex relabelling. 

A graph is called \emph{2-connected} if at least two vertex deletions are needed in order to disconnect it.

A graph is an $m$-\emph{cycle}, denoted $C_m$, if $E=\{\{v_m,v_1\},\{v_1,v_2\},...,\{v_{m-1},v_m\}\}$ for $m> 2$ and some ordering $v_1,...,v_m$ of $V$. Let $C_m$ be a subgraph of $G$.  Any edge $\{v_i,v_j\}\in E_G,$ such that $\{v_i,v_j\}\subset V_{C_m}$ and $\{v_i,v_j\}\not\in E_{C_m}$ is called a \emph{chord} of $C_m$. If $V_{C_m}=V$, $C_m$ is called a \emph{Hamilton cycle} of $G$.

Suppose that $G$ admits a \emph{planar} embedding $\Gamma $ on the plane, i.e. an embedding such that the edges do not cross one another. The closures of the connected components of $\mathbb{R}^2\backslash \Gamma $ are called \emph{faces} of the embedding and there is always a unique face that is unbounded. Edges that lie on this face will be called \emph{outer}; otherwise, they will be called \emph{inner}.

\paragraph{Outerplanar Graphs.}
Let $P$ be a polygon with vertices numbered in $\{1,...,n\}$, in counterclockwise order. A \emph{polygon dissection} is an arrangement of diagonals on $P$, such that no two of them are intersecting.

\emph{Outerplanar graphs} are graphs that can be embedded on the plane in such a way that all vertices lie on the boundary of the unbounded face. Let $\mathcal{G}$ be the class of labelled outerplanar graphs. In \cite{bodirsky2007enumeration}, the authors bring together a set of combinatorial constructions that define the class $\mathcal{G}$ and involve the classes $\mathcal{D}, \mathcal{B}, \mathcal{C}$, corresponding to polygon dissections, labelled 2-connected, and labelled connected outerplanar graphs, respectively. These constructions translate to the functional equations in Table~\ref{table:2}. Note that the first one is an ordinary generating function and the rest are exponential.

\begin{table}[h!]\centering
\ra{1}
\begin{tabular}{@{ }ll@{ }}

Polygon dissections &  $D(z)=z/4+z^2/4-z/4\,\sqrt {{z}^{2}-6\,z+1} $\\ 

2-connected outerplanar   &   $B'(z)=\frac{1}{2z}D(z)+\frac{z}{2}$\\ 

Connected outerplanar    &  $zC'(z)=z\exp (B'(zC'(z)))$ \\ 

General outerplanar     &   $G(z)=\exp (C(z))$ \\

\end{tabular}
\caption{A defining set of equations for labelled outerplanar graphs}\label{table:2}
\end{table}
We are interested on how the first equation is derived. Each 2-connected outerplanar graph with $|V|>2$ has a Hamilton cycle, so, assuming a numbering on it, counting 2-connected outerplanar graphs of size $n$ is equivalent to counting dissections of the same size. The starting point is \cite{flajolet1999analytic}, where polygon dissections were modelled in a symbolic way, based on the recursive structure of the class, as shown in Figure~\ref{fig:1}. In short, one designates an edge of the polygon, say the $e=\{v_1,v_2\}$ edge, and then divides the dissections according to the size of the polygon where $e$ lies. The latter will be called \textit{root polygon} and $e$ will be called \emph{root}. On the rest of the edges, other dissections are attached.

\begin{figure}[h!]\centering
  \includegraphics[scale=.75]{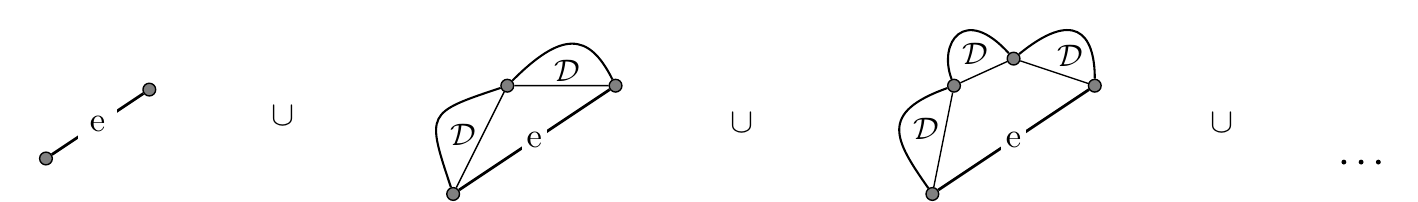}
  \caption{A combinatorial decomposition of fixed-polygon dissections.}\label{fig:1}
\end{figure}

\noindent The following translation is then implied, in terms of ogfs:
\begin{equation}D=z^2+\frac{D^2}{z}+\frac{D^3}{z^2}+\frac{D^4}{z^3}+...=z^2+\frac{D^2}{z-D}\Rightarrow 2D^2-D(z+z^2)+z^3=0.\label{eq:1}\end{equation}

The second equation in Table~\ref{table:2} is derived by observing that $n![z^n]B(z)=\frac{(n-1)!}{2}[z^n]D(z)$. The third and fourth correspond to the symbolic constructions: $\mathcal{Z}\star \mathcal{C}^{\circ}=\mathcal{Z}\star \mathrm{S\scalebox{.9}{ET}}(\mathcal{B}^{\circ}\circ \mathcal{C}^{\bullet})$ and $\mathcal{G}=\mathrm{S\scalebox{.9}{ET}}(\mathcal{C})$. The former relation is well-known (see for instance \cite[p.10]{harary9graphical},\cite{gimenez2009asymptotic}, \cite{subcritical}) and is based on the decomposition of a graph into 2-connected components. The latter one is straightforward.

\paragraph{Analytic Preliminaries.}
\noindent We denote by $\mathbf{y}=\mathbf{F}(z,\mathbf{y},\mathbf{u})$ a system of the form:
 \[
    \left\{
                \begin{array}{lll}
y_1&=& f_1(z,\mathbf{y},\mathbf{u})\\
\vdots & & \vdots \\
y_{m} &=& f_{m}(z,\mathbf{y},\mathbf{u}).
\end{array}
              \right. 
  \]
Let $f_1,...,f_m$ be analytic functions with non-negative coefficients, such that $\mathbf{F}(0,\mathbf{0},\mathbf{u}) \equiv \mathbf{0}$, $\mathbf{F}(z,\mathbf{0},\mathbf{u}) \not\equiv \mathbf{0}$ and there exists $j$ with $\mathbf{F}_{y_jy_j}\not\equiv \mathbf{0}$, where $\mathbf{F}_{y_jy_j}$ denotes the second derivative with respect to $y_j$. To any such system, we relate a directed graph with vertices $y_i$ and edges $(y_i,y_j)\in E$ whenever $F_i$ depends on $y_j$, i.e. whenever $\frac{\partial F_i}{\partial y_j}\not\equiv 0$. We call this the \emph{dependency graph} of the system and suppose that it is strongly connected, i.e. one can move from any vertex to any other through a path of directed edges. If such a system has unique analytic solutions with non-negative coefficients $y_i(z,u_1,...,u_{m})$ around $z=0,u_i=1$, it is called  \emph{well defined}. Then, Theorem~\ref{drmota} \cite[Prop.3]{drmota1997systems}\cite[Ch.2]{drmota2009random} holds, adjusted to our purpose: the only missing requirement is $\mathbf{F}(0,\mathbf{y},\mathbf{u})=0$, but the result is still valid when one deals with well-defined systems.

\begin{theorem}\label{drmota}
Let $\mathbf{y}=\mathbf{F}(z,\mathbf{y},\mathbf{u})$ be a well-defined system and let $\mathbf{y}=\mathbf{y}(z,\mathbf{u})=(y_1(z,\mathbf{u}),...,y_N(z,\mathbf{u}))$ denote the analytic solutions of the system. Suppose that the radius of convergence of $\mathbf{F}$ is large enough that there is a positive number $z_0$ of minimum modulus and real numbers $\mathbf{y}_0$ that satisfy the system \begin{gather}\label{charsystem}\begin{split}\mathbf{y} &= \mathbf{F}(z,\mathbf{y},\mathbf{1}) \\
0 &= \det (\mathbf{I}-\mathbf{F}_{\mathbf{y}}(z,\mathbf{y},\mathbf{1})).\end{split}\end{gather} Then there exist functions $\rho (\mathbf{u}), g_i(z,\mathbf{u}), h_i(z,\mathbf{u}),$ for $1\leq i\leq N,$ which are analytic around $z=z_0,\mathbf{u}=\mathbf{1},$ and satisfy $\rho (\mathbf{1})=z_0,$ $h_i(z_0,\mathbf{1})<0,$ such that:
\begin{equation} y_i(z,\mathbf{u})=g_i(z,\mathbf{u})-h_i(z,\mathbf{u})\sqrt{1-\frac{z}{\rho (u)}}\label{criticalexp}\end{equation}
 locally around $z=z_0$, $\mathbf{u}=\mathbf{1}$ with $\arg (z-\rho (\mathbf{u}))\neq 0.$  Assume also that $[z^n]y_j(z,\mathbf{1})>0$ for $1 \leq j \leq N$ and for all large enough $ n$. Then, for $\mathbf{u}$ sufficiently close to $\mathbf{1}$, the radius of convergence of all $y_i$ is $\rho (u)$ and there are no other singularities on the circle of convergence $|z|=|\rho (u)|$ than $z=\rho (u)$. Furthermore, there exists $\epsilon >0$, such that $y_i$ can be analytically continued to the region $|z| < |\rho (u)|+\epsilon $, $|\arg(z-\rho (u))|>\epsilon $. \label{eq:expansion}\end{theorem}

\noindent Note that, according to Condition~\eqref{charsystem}, 1 is an eigenvalue of the matrix $\mathbf{F}_{\mathbf{y}}(z_0,\mathbf{y}_0,\mathbf{1})$. Systems like~\eqref{charsystem} will be called \emph{characteristic}. In expansions of the form~\eqref{criticalexp}, we will call \emph{critical exponent} the first non-integer power of the expansion (in this case, for instance, the critical exponent is equal to $1/2$).

For the asymptotic analysis, we follow the transfer principles of \emph{singularity analysis}, as they are presented in \cite{flajolet2009analytic}. Let $f(z)$ be an analytic function at zero with a unique smallest singularity at $z=\rho $ and $\rho >0$. We need the fact that, if $f(z)$ has a singular expansion $f(z)=a_0+a_1(1-z/ \rho )^{-\alpha }+\mathcal{O}\big( (1-z/ \rho )^{-\alpha +\delta })\big)$ in a domain $|z| \leq \rho+\epsilon $, $|z-\rho |\geq\epsilon $, where $\delta , \epsilon >0$ and $\alpha \in \mathbb{C}\backslash \mathbb{Z}_{\leq 0}$, then: \[[z^n]f(z)=a_1\frac{n^{\alpha -1}}{\Gamma (\alpha )}\rho ^{-n} \big(1+o(1)\big) ,\] where $\Gamma (\alpha )$ refers to the \emph{Euler Gamma function}, defined as $\Gamma (x)=\int _0^\infty t^{x-1}e^{-t}dt$.  

For the extraction of normal limit laws, we use Theorem 2.25 from~\cite{drmota2009random}.

\begin{theorem} Suppose that a sequence of $k$-dimensional random vectors $\mathbf{X}_n$ satisfies 
$\mathbb{E} [\mathbf{u}^{\mathbf{X}_n}]=\frac{c_n(\mathbf{u})}{c_n(\mathbf{1})},$ where $c_n(\mathbf{u})$ is the coefficient of $z^n$ of an analytic function $y(z,\mathbf{u})=\sum _{n\geq 0}c_n(\mathbf{u})z^n$ around $z=0,\mathbf{u}=\mathbf{1}$ and $c_n(\mathbf{u})>0$ for $n\geq n_0$ and positive real $\mathbf{u}$. Suppose also that $y(z,\mathbf{u})$ has a local singular representation of the form \[y(z,\mathbf{u})=g(z,\mathbf{u})+h(z,\mathbf{u})\Big( 1-\frac{z}{\rho (\mathbf{u})}\Big)^{\alpha} \] for some real $\alpha \in \mathbb{R}\backslash \mathbb{N}$ and functions $g(z,\mathbf{u}),h(z,\mathbf{u})\neq 0$ and $\rho (\mathbf{u})\neq 0$ that are analytic around $z=z_0>0$ and $\mathbf{ u}=\mathbf{1}$. If $z=\rho (\mathbf{u})$ is the only singularity of $y(z,\mathbf{u})$ on the disk $|z|\leq |\rho (\mathbf{u})| $, when $\mathbf{u}$ is sufficiently close to $\mathbf{1}$, and there exists an analytic continuation of $y(z,\mathbf{u})$ to the region $|z|< |\rho (\mathbf{u})| +\delta$, $|\arg (z-\rho (\mathbf{u}))|>\epsilon $ for some $\delta >0$ and $\epsilon >0$, 
then $\mathbf{X}_n$ satisfies a central limit theorem \[\frac{1}{\sqrt{n}}(\mathbf{X}_n-\mathbb{E}[\mathbf{X}_n])\xrightarrow[]{d} N(\mathbf{0},\mathbf{\Sigma })\] with 
\[\mathbb{E}[\mathbf{X}_n]=\boldsymbol{\mu }n+\mathcal{O}(1)\quad \mathrm{ and }\quad \mathbb{C}\textsl{ov}\hspace{.03cm}[\mathbf{X}_n]=\mathbf{\Sigma}n+\mathcal{O}(1),\] where \[\boldsymbol{\mu }=-\frac{\rho _{\mathbf{u}}(\mathbf{1})}{\rho (\mathbf{1})}\quad \mathrm{ and } \quad\boldsymbol{\Sigma } =-\frac{\rho _{\mathbf{uu}}(\mathbf{1})}{\rho (\mathbf{1})}+\boldsymbol{\mu}\boldsymbol{\mu}^T +\textsl{diag}( \boldsymbol{\mu } ).\]

 \label{quasi}
\end{theorem}
Finally, a pair of combinatorial classes with generating functions $(y(z),g(z))$ is called \textit{subcritical} if $y(z)=g(y(z))$ and $y(\rho _y )<\rho _g $, where $\rho _y$ and $\rho _g $ are the radius of convergence of $y,g$, respectively.

\section{Encoding 2-connected patterns in  polygon dissections} \label{maain}
Let $\Delta =\{ \delta _1,\delta _2,...,\delta _m \} $ be a set of 2-connected patterns and let $D_{\Delta }(z,\mathbf{u})$ be a multivariate generating function, where the coefficient of $z^nu_1^{n_1}\cdots u_m^{n_m}$ is the number of polygon dissections in $\mathcal{D}$ that have $n$ vertices and $n_i$ subgraph occurrences of the pattern $\delta _i$.\footnote{From now on, we will refer to $\delta _i$ also as \emph{parameters}, in an abuse of terminology, since we are interested in their number in polygon dissections of size $n$.} In the construction of Figure~\ref{fig:1} and the corresponding Equation~\eqref{eq:1}, observe that the encoding of subgraphs of type $\delta _i$ is not straightforward, since they do not behave additively as parameters. The aim in this section is to prove, for any set $\Delta $, an explicit combinatorial construction for $\mathcal{D}$ that allows this encoding. The approach we follow is to partition the class $\mathcal{D}$ into smaller combinatorial classes and build a symbolic system with them and $\mathcal{D}$, in which we can handle uniformly the appearance of such patterns. The resulting system uses only the operations of addition and cartesian product, and thus settles formally the algebraic nature of $D_{\Delta }(z,\mathbf{u})$. 

For clarity of presentation, we first work with dissections $\bar{\mathcal{D}}$ that miss one of the two vertices of the root-edge, hence $[z^n]D(z,\mathbf{1})=[z^{n-1}]\bar{D}(z,\mathbf{1})$, in order to avoid the denominators of Equation~\eqref{eq:1}. We proceed by defining the auxiliary combinatorial classes $\bar{\mathcal{D}}_{\circ },\bar{\mathcal{D}}_{\nu _1},...,\bar{\mathcal{D}}_{\nu _k}$ that give us the required partition.

Since $\delta _i$ are subgraphs of polygon dissections and 2-connected, they are themselves isomorphic to polygon dissections. Let $H_{\Delta  }$ be the length of the maximum Hamilton cycle over all $\delta _i$. In order to encode the appearances of $\delta _i$ in $\bar{\mathcal{D}}$, we need to control the way the dissections are glued recursively, as suggested in Figure~\ref{fig:1}. In fact, we need to control the construction until the root polygon is of length at most $H_{\Delta  }$: no new copies of $\delta _i$ are created when already made dissections are glued around a big root polygon. Thus, we consider as \emph{small}, respectively \emph{big}, the polygons that are equal to or smaller than, respectively larger than, an $H_{\Delta  }$-gon and denote by $\bar{\mathcal{D}}_{\circ  }$ the class which contains all dissections in $\bar{\mathcal{D}}$ that have a big root polygon, plus the edge $e=\{v_1,v_2\}$. We are now able to give the following definition:


\begin{definition} A polygon dissection is called a \emph{composite root} with respect to a set of 2-connected patterns $\Delta =\{ \delta _1,\delta _2,...,\delta _m \} $ if the following two conditions hold:
\begin{enumerate} \item It consists only of small polygons.
\item Let $F$ be a face of the composite root that shares an edge with the unbounded face and let an edge $e_1\in F$ that is not an outer edge. Then, $e_1$ is connected to the root-edge with a simple path of adjacent polygons that constitutes a dissection of size less than $H_{\Delta }$.
\end{enumerate} \end{definition}


Observe that there is a finite number of composite roots. They are denoted  by $\nu _j$, where $j$ refers to some arbitrary ordering among them. Alternatively, we identify a composite root with a tuple of indices $i_1[i_2]$, where the first index $i_1$ is the size of its root polygon and the second index $i_2$ is its ordering among all the other composite roots with the with the same root polygon of size $i_1$, according to some arbitrary ordering among them (see, for instance, Figure~\ref{new444} or~\ref{fig:2}).

Let $A,B$ be polygon dissections. $B$ will be called an \emph{extension} of $A$ if $A$ is an induced subgraph of $B$ preserving the root-edge, i.e., one can obtain $A$ from $B$ by a sequence of vertex deletions, excluding the vertices of the root-edge, and renumbering the vertices according to their final position with respect to the root-edge. For instance, the dissections $3[3]$ and $3[8]$ in Figure~\ref{new444} are extensions of $3[1]$, but $3[9]$ is not an extension of $4[1]$. A composite root is called \emph{maximal} if there is no composite root that is an extension of it.


\begin{figure}[h!]\centering
  \includegraphics[scale=7]{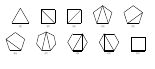}
  \vspace{-1cm}
  \caption{The composite roots, when $\Delta =\{\delta _1\}$ and $\delta _1$ is a 4-cycle. The roots $3[4]$, $3[7]$, $3[8]$, $3[9]$, and $4[1]$ are the only maximal ones and an edge is blue if it is outer in some maximal extension.}\label{new444}
\end{figure}

We associate to each one of the composite roots $\nu _j$ the combinatorial class $\bar{\mathcal{D}}_{\nu _j},$ which corresponds to polygon dissections that are extensions of the composite root $\nu _j$ and satisfy the following condition, called \emph{Condition (I)}: \begin{enumerate}
\item[(I)] If an outer edge of $\nu _j$ is inner in the maximal extensions of $\nu _j$, then only elements of $\bar{\mathcal{D}}_{\circ  }$ are attached on it.
\end{enumerate}

\begin{lemma}
The classes $ \bar{\mathcal{D}}_{\circ},\bar{\mathcal{D}}_{\nu _j}$ partition $\bar{\mathcal{D}}$. Moreover, each of the classes $\bar{ \mathcal{D}}, \bar{\mathcal{D}}_{\circ},\bar{\mathcal{D}}_{\nu _j}$ can be constructed in a non-trivial admissible way by the classes $\bar{ \mathcal{D}}, \bar{\mathcal{D}}_{\circ},\bar{\mathcal{D}}_{\nu _j}$.\end{lemma}
\begin{proof}
Condition (I) forces the $\bar{\mathcal{D}}_{\nu _j}$ classes to be disjoint: if $p _i\in \bar{\mathcal{D}}_{\nu _i}, p _j\in \bar{\mathcal{D}}_{\nu _j}$, $i\neq j$, then $p _i\neq  p _j$, since their maximal composite roots differ in at least one small polygon. Moreover, any object in $\bar{\mathcal{D}}$ with small root polygon and maximal composite root $\nu _j$ belongs in $\bar{\mathcal{D}}_{\nu _j}$. Since the class $\bar{\mathcal{D}}_{\circ  }$ contains the edge graph and all dissections with big root polygon, the $\bar{\mathcal{D}}_{\circ},\bar{\mathcal{D}}_{\nu _j}$ classes indeed form a partition of $\bar{\mathcal{D}}$.  It then holds that

 \begin{equation}\bar{\mathcal{D}}=\bar{\mathcal{D}}_{\circ  }\text{ }\bigcup _{j=1 }^m\bar{\mathcal{D}}_{\nu _j } \quad\text{ and }\quad\bar{\mathcal{D}}_{\circ  }=\{ \includegraphics[scale=.6]{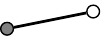}\}\bigcup _{i > H_{\Delta  }}\underbrace{\bar{\mathcal{D}}\times  \ldots
            \times\bar{\mathcal{D}}}_{i \text{ times}}.  \label{union}\end{equation}
            
            For $p\in \bar{\mathcal{D}}_{\nu _j}$, $p$ is decomposed uniquely into its maximal composite root $ \nu _j$ and a sequence of objects from the classes $\bar{\mathcal{D}}_{\circ},\bar{\mathcal{D}}_{\nu _j}$ that respects Condition (I). In particular, if an edge of $\nu _j$ is outer in its maximal extensions, then objects from any class are attached. Else, only members of $\bar{\mathcal{D}}_{\circ  }$ are attached. Let $t$ be the number of such outer edges in $\nu  _{j}$, $s$ be the number of all outer edges, and $\mathbf{c}\in 
\{\circ  ,\nu _1,\nu _2,...,\nu _m\} ^t $. Then it holds that\vspace{0cm} \begin{equation}\bar{\mathcal{D}}_{{\nu _j}}=\bigcup _{\mathbf{c}} \underbrace{\bar{\mathcal{D}}_{\circ  } \times  \ldots   \times\bar{\mathcal{D}}_{\circ  }}_{s-t-1 \text{ times}}\times \bar{\mathcal{D}}_{c_1}\times ...\times  \bar{\mathcal{D}}_{c _t}\label{unionplus}.\end{equation}
\end{proof}

\begin{figure}[h!]\centering
  \includegraphics[scale=7]{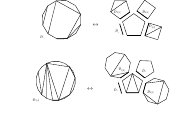}
  \vspace{-1cm}
  \caption{The decomposition of an element in  $\bar{\mathcal{D}}_{\circ }$ and an element in $\bar{\mathcal{D}}_{3[4]}$, when $\Delta =\{\delta _1\}$ and $\delta _1$ is a 4-cycle. See Figure~\ref{new444} for the class indices.}\label{dec1}
\end{figure}

\begin{theorem}
\label{main}
Let $\Delta =\{ \delta _1,\delta _2,...,\delta _m \} $ be a set of 2-connected patterns and $\nu _1 ,..., \nu _k$ the corresponding composite roots. The generating functions $\bar{D}(z,\mathbf{u}), \bar{D}_{\circ  }(z,\mathbf{u}),\bar{D}_{\nu_1}(z,\mathbf{u}),...,\bar{D}_{\nu _k}(z,\mathbf{u}),$ where $\mathbf{u}=(u_1,...,u_m),$ satisfy a computable system of the form:
 \[
    \left\{
                \begin{array}{lll}
y&=&r(z,u_1,...,u_m,y,y_{\circ },y_{\nu _1},...,y_{\nu _k}),\\
y_{\circ }&=&r_{ \circ}(z,u_1,...,u_m,y,y_{\circ  },y_{\nu _1},...,y_{\nu _k}),\\
y_{\nu _1} &=& r_{\nu _1}(z,u_1,...,u_m,y,y_{\circ  },y_{\nu _1},...,y_{\nu _k}),\\
  \vdots & & \vdots \\
y_{\nu _k} &=& r_{\nu _k}(z,u_1,...,u_m,y,y_{\circ  },y_{\nu _1},...,y_{\nu _k}).
\end{array}
              \right.
  \]
which is non-linear in $y,y_{\circ  },y_{\nu _j}$ and where each $r_j$ is a $\mathbb{Q}$-rational and analytic function around zero, with non-negative coefficients. Moreover, $r_{\circ  }(z,\mathbf{0})\neq 0$ and the system is strongly connected.

\end{theorem}

\begin{proof}  
\noindent The parameters $\delta _i$ are additive in the symbolic Equations~\eqref{union}, so their translation to multivariate generating functions depending on $z$ and $u$ is immediate: \begin{equation}\bar{D}=\bar{D}_{\circ  }+\sum _{i=1}^m \bar{D}_{\nu _i}, 
\quad\quad \bar{D}_{\circ  }=z+\sum _{i> H_{\Delta  }}\bar{D}^i\Rightarrow \bar{D}_{\circ  }=z+\frac{\bar{D}^{H_{\Delta }}}{1-\bar{D}}.\label{eq:1234}\end{equation} 

The parameters $\delta _i$ are not additive in the symbolic Equation~\eqref{unionplus}, since new copies of them might occur after the attachment of objects in $\bar{\mathcal{D}}_{{\nu _j}}$ around the composite root. However, any new copies occur locally, in the interactions between $\nu _j$ and subsets of $\mathbf{c}$. This is a fixed number $p_{\mathbf{c}}^{ji}$ for every $\mathbf{c}$ and $\delta _i$. Thus, Equation~\eqref{unionplus} is translated in the following way:
            \begin{equation} \bar{D}_{\nu _j}=\sum _{\mathbf{c} } \bar{D}_{\circ  }^{s-t-1}\bar{D}_{c _1}...\bar{D}_{c _t}u_1^{p_{\mathbf{c}}^{j1}} ...u_m^{p_{\mathbf{c}}^{jm}}.\end{equation}           
  
The emerging system is indeed strongly connected: all $\bar{D}_{\nu _j}$ are connected to $\bar{D}_{\circ }$, which connects to $\bar{D}$. The rest of the stated properties are immediate.
  
\end{proof}
 
By Equation~(\ref{eq:1234}), we obtain $\bar{D}_{\circ}=\bar{D}^{H_{\Delta }}+\bar{D}\bar{D}_{\circ}-z\bar{D}+z$. Also, one can substitute the $y_{\circ  },y_{\nu _j}$ variables in the right part of $y$'s equation with their equivalent expressions. Thus, systems of Theorem~\ref{main} can be turned to \textit{proper algebraic}, i.e., in the right part there is no constant term or linear term $y,y_i$. Then, we can argue that $\bar{D}(z,\mathbf{u})$ also satisfies some computable polynomial equation $p(z,\mathbf{y},\mathbf{u})=0$ (see \cite{Panholzer}). Consequently, also $D(z,\mathbf{u})$ is algebraic, as well as $D(z,\mathbf{0})$, i.e., the generating function of polygon dissections that avoid all patterns in $\Delta $ as subgraphs. 

\begin{corollary}
The generating function $D(z,\mathbf{u})$ is algebraic and the defining polynomial is computable. The generating function of polygon dissections that avoid all $\delta $-patterns as subgraphs, $D(z,\mathbf{0})$, is likewise algebraic.\label{cor}
\end{corollary}

Note that the systems resulting from Theorem~\ref{main} are large with respect to $H_{\Delta  }$. In particular, any combination of at most $H_{\Delta  }-2$ small polygons around a root polygon of size $H_{\Delta  }-1$ will constitute a composite root. These are  $(H_{\Delta  }-1)^{H_{\Delta  }-2}$, since there are $H_{\Delta  }-2$ available edges and $H_{\Delta  }-1$ choices, when considering also the empty choice. However, when $H_{\Delta  }$ is small, one can find ad hoc arguments to make the systems manageable; see for instance Section~\ref{app}.



\begin{theorem}\label{random}
Let $\Delta =\{\delta _1,...,\delta _m\}$ be a set of 2-connected patterns. Let $\Omega _n$ be the set of polygon dissections of size $n$ and $X_n:\Omega _n \rightarrow \mathbb{Z}_{\geq 0}^m$ be a vector of random variables $X_{\delta _1},...,X_{\delta _m}$ in $\Omega _n$, such that $X_{\delta _i}(\omega )$ is the number of $\delta _i$ patterns in $\omega \in\Omega _n$. Then, $\mathbf{X}_n$ satisfies a central limit theorem \[\frac{1}{\sqrt{n}}(\mathbf{X}_n-\mathbb{E}[\mathbf{X}_n])\xrightarrow[]{d} N(\mathbf{0},\mathbf{\Sigma })\] with 
\[\mathbb{E}[\mathbf{X}_n]=\boldsymbol{\mu }n+\mathcal{O}(1) \text{ and } \mathbb{C}\textsl{ov}\hspace{.03cm}[\mathbf{X}_n]=\mathbf{\Sigma}n+\mathcal{O}(1),\] where $\boldsymbol{\mu }$ and $\mathbf{\Sigma}$ are computable. 
\end{theorem}
\begin{proof}
Any system resulting from Theorem~\ref{main}, $\mathbf{y}-\mathbf{r}(\mathbf{y},\mathbf{z},\mathbf{u})=\mathbf{0}$, admits a non-negative power series solution $\mathbf{y}(z,\mathbf{u})$ around zero and $\mathbf{1}$ by construction. This is also unique by the implicit function theorem, since \[\det (\mathbf{I}-\mathbf{r}_{\mathbf{y}}(\mathbf{y},z,\mathbf{u}))| _{(\mathbf{y},z)=\mathbf{0},\mathbf{u}=\mathbf{1}}=1\] for every set $\Delta $. Thus, the defining system of $\bar{D}$ is always well defined. 

By construction, the system is also strongly connected. Consequently, the matrix $\mathbf{r}_{\mathbf{y}}$ is non-negative and irreducible in $\mathbb{R}^+$. It is known~\cite{minc1988nonnegative} that non-negative irreducible matrices have a unique dominant eigenvalue $\lambda $ that is positive and strictly increasing with respect to the entries of the matrix. Let $\rho $ be the radius of convergence of $\bar{D}(z,\mathbf{1})$. For $z< \rho $, it holds that $ \lambda (\mathbf{r}_{\mathbf{y}} (z, \mathbf{y}(z),\mathbf{1} ))<1$: if this was not the case, then $\bar{D}(z,\mathbf{1})$'s radius of convergence would be smaller, by Theorem~\ref{drmota}. The value $\bar{D} (\rho ,\mathbf{1})$ is finite, since $\bar{D}$ is algebraic. Consequently, the characteristic system always has the minimal solution $(\rho ,y(\rho ,\mathbf{1}))$. Moreover, it is true that $[z^n]\bar{D}(z,\mathbf{1})>0$.

The result can now be obtained as direct consequence of Theorems~\ref{main}, \ref{drmota}, and \ref{quasi}.

\end{proof}

\section{Applications }\label{app}

In this section, we give examples and applications of Theorem~\ref{main}. The applications concern the combinatorial classes of polygon dissections and outerplanar graphs and they are of two different kinds: computation of limit laws for 2-connected parameters $\delta _i$ and asymptotic enumeration of these classes, when the patterns $\delta _i$ are forbidden as subgraphs. For clarity, we give the defining equations for $\bar{D}$ and not for $D$, but the final computations will be done in terms of $D$. The equation analysis process is similar to the one in~\cite{bodirsky2007enumeration}.

\subsection{Extraction of limit laws}

\subsubsection*{Encoding $3$-cycles}

\noindent The only composite root is the triangle, denoted by $3[1]$. Thus, the defining system of $\bar{D}$ is the following:
\begin{eqnarray*}\bar{D} &=& \bar{D}_{\circ  }+\bar{D}_{3[1]},\\ \bar{D}_{\circ  } &=& z+\frac{\bar{D}^{3}}{1-\bar{D}},\\ \bar{D}_{3[1]} &=& u(\bar{D}_{\circ  }+\bar{D}_{3[1]})^2.\end{eqnarray*}
The latter is equivalent to the following polynomial system (notice that in this form it is not non-negative):
\begin{eqnarray*}\bar{D} &=& \bar{D}_{\circ  }+\bar{D}_{3[1]},\\ \bar{D}_{\circ  } &=& \bar{D}^3+\bar{D}\bar{D}_{\circ  }-z\bar{D}+z,\\ \bar{D}_{3[1]} &=& u(\bar{D}_{\circ  }+\bar{D}_{3[1]})^2.\end{eqnarray*}
\noindent By observing that $\bar{D}_{3[1]}=u\bar{D}^2$ and $\bar{D}\bar{D}_0=\bar{D}(\bar{D}-\bar{D}_{3[1]})=\bar{D}^2(1-u\bar{D})$, we obtain
\[\bar{D}=\bar{D}^3(1-u)+\bar{D}^2(1+u)-\bar{D}z+z.\]

\subsubsection*{Encoding $4$-cycles}

\noindent The composite roots are all the dissections in Figure~\ref{new444}. From now on, we write $\bar{D}_i$ for the sum $\sum _j \bar{D}_{i[j]}$. Also, when $m$ equations are the same and correspond to the same root polygon with $n$ sides, we write $ \bar{D}_{n[i_1,...,i_m]}$ or $ \bar{D}_{n[i_1-i_m]}$, for shortness.

\begin{minipage}{.4\textwidth }\begin{eqnarray*}\bar{D} &=& \bar{D}_{\circ  }+\bar{D}_{ 3}+\bar{D}_{ 4},\\ \bar{D}_{\circ  } &=& \bar{D}^4+\bar{D}\bar{D}_{\circ  }-z\bar{D}+z, \\ \bar{D}_{3[1]} &=& \bar{D}_{\circ  }^2 ,\\ \bar{D}_{3[2,3] } &=& u\bar{D}_{\circ  }(\bar{D}_{\circ  }+u\bar{D}_{3}+\bar{D}_{4[1]})^2 ,
\end{eqnarray*}\end{minipage}
\begin{minipage}{.6\textwidth }
\begin{eqnarray*}   \bar{D}_{3[4]} &=& u^2(\bar{D}_{\circ } +u \bar{D}_{3} +\bar{D}_{4[1]})^4 , \\ \bar{D}_{3[5,6]} &=& u\bar{D}_{\circ }\bar{D}^3 , \\ \bar{D}_{3[7]} &=& u^2\bar{D}^6  , \\ \bar{D}_{3[8,9]} &=& u^2\bar{D}^3(\bar{D}_{\circ }+u\bar{D}_3+\bar{D}_{4[1]})^2 , \\ \bar{D}_{4[1]} &=& u\bar{D}^3.
\end{eqnarray*}
\end{minipage} 

\vspace{.5cm}\noindent Notice that the term $(\bar{D}_{\circ } +u \bar{D}_{3} +\bar{D}_4)^2$ is equal to $\bar{D}_3$. So, the system is equivalent to:

\[\bar{D}=\bar{D}_{\circ  }+\bar{D}_3+\bar{D}_{4[1]} ,\quad \bar{D}_{\circ  }= \bar{D}^4+\bar{D}\bar{D}_{\circ  }-z\bar{D}+z ,\quad \bar{D}_3=(\bar{D}_{\circ  }+u\bar{D}_3+\bar{D}_{4[1]} )^2 \quad \bar{D}_{4[1]} =u\bar{D}^3.\]

\vspace{.5cm}
\noindent  We now use the previous systems, encoding $3$ and $4$-cycles, to obtain the following theorem:

\begin{theorem}\label{random2}
The number of appearances $X_n$ of 3-cycles and 4-cycles in polygon dissections and outerplanar graphs of size $n$ follows a central limit theorem as in \ref{1st}, where the mean and variance are asymptotically linear. The constants are the following, in their exact values for dissections and in approximation for outerplanar graphs:
\begin{table}[h!]\centering
\begin{tabular}{l|ll|ll}
Parameter  &  $\mu $ & $\sigma ^2 $ &  $\mu $ & $\sigma ^2 $\\
\midrule[.8pt]

3-cycles   &$\frac{1}{2}$ & ${\frac {-13+9\,\sqrt {2}}{-12+8\,\sqrt {2}}}\approx 0.39644$&  0.34793 &  0.40737  \\ 

4-cycles     & $ {\frac {-30+21\,\sqrt {2}}{-12+8\,\sqrt {2}}}

\approx 0.43933$ & $\,{\frac {-24216+17123\,\sqrt {2}}{-32 \left( -3+2\,\sqrt {2}
 \right) ^{2}}}
\approx 0.44710$ &  0.33705  &  0.36145\\ 

\end{tabular}
\end{table}
\end{theorem}

\begin{proof}
The central limit theorem is obtained from Theorem~\ref{random}. We present an outline of how to get the exact constants in both cases, which can be replicated for any system derived from Theorem~\ref{main}. For specific steps of the computations, see Section~\ref{comp}

In both cases, $D$ has a singular expansion of the form  \[g(z,u)-h(z,u)\sqrt{1-\frac{z}{r (u)}},\] that satisfies the requirements of Theorem~\ref{quasi}. This can be obtained by the same reasoning as in Theorem~\ref{random}.

The value $r(1)$ can be computed using the \textit{discriminant} of $D$'s defining polynomial $p(y,z,u)$ (see~\cite[Ch.VII]{flajolet2009analytic}), $\mathrm{disc}(z,u)$. In this case, $r(1)=3-2\sqrt{2}$, which is known from~\cite{flajolet1999analytic}. Then, we also find the values $r'(1), r''(1)$, by consecutively differentiating $\mathrm{disc}(r(u),u)$ with respect to $u$. With these values, we compute the constants required for the mean and variance according to Theorem~\ref{quasi} and obtain the indicated numbers. 

In order to pass to labelled 2-connected, connected, and then general outerplanar graphs, we use the multivariate analogues of the equations in Table~\ref{table:2}, i.e.
\[B'(z,u)=\frac{1}{2z}D(z,u)+\frac{z}{2}, \quad zC'(z,u)=z\exp (B'(zC'(z,u),u)), \quad G(z,u)=\exp (C(z,u)),\]
where the derivatives are taken with respect to $z$. Let $y$ denote $zC'(z,1)$ and consider the characteristic system:
\[y-z\exp (B'(y,1))=0\]
\[1-z\exp (B'(y,1))B''(y,1)=0\]
This has indeed a minimal positive solution $(\tau ,z_0)$, since the outerplanar graph class belongs to the subcritical family of graphs~\cite{subcritical}, i.e.  $z_0 C'(z_0,1)<r(1)$, where $z_0$ is the radius of convergence of $C'(z,1)$ and $r(1)$ the one of $B'(z,1)$.  Moreover, the system satisfies $1-yB''(y,1)=0$. Solving for $y$, we find the value $\tau $ and then $z_0=\tau \exp (-B'(\tau ,1))$. 

We can now apply Theorem \ref{drmota} and get a singular expansion around $z_0$ for $C'$, with critical exponent $1/2$. Moreover, the point $z_0$ is the only singularity on the radius of convergence of $C'$ and there exists an analytic function $\rho (u)$ around $u=1$ that gives the unique smallest singularity of $C'$ when $u$ is close to $1$; in particular, $\rho (1)=z_0$. Then, also $\tau (u)$ is an analytic function close to 1, where $\tau (u)=\rho (u)C'(\rho (u),u)$. 

 As in \cite{bodirsky2007enumeration}, if $\Psi (y,u)$ is an analytic function such that $\Psi (y,u)=y\exp (-B'(y,u))$, then $\rho (u)=\Psi (\tau (u),u)$ and it holds that \begin{equation}\rho '(u)=\frac{\partial \Psi }{\partial u}(\tau (u),u)\quad \text{ and }\quad \rho ''(u)=\frac{\partial ^2\Psi }{\partial y\partial u}(\tau (u),u)\tau '(u)+\frac{\partial ^2\Psi}{\partial u^2}(\tau (u),u).\label{coomp}\end{equation}
 

The functions $C$ and $G$ have the same singularity function $\rho (u)$ as $C'$, but the critical exponent of their expansion on $\rho (u)$ is $3/2$ (see the analysis in \cite{bodirsky2007enumeration} for details). We can thus apply Theorem \ref{quasi}, after computing the relevant constants. The value $\tau '(1)$ can be computed through the relation $\tau (u)B''(\tau (u),u)=1$. 

\end{proof}

In the case of outerplanar graphs, the limit laws are expected from \cite{subcritical}. However, in \cite{subcritical} there is no constructive way to compute the relevant constants. This is a contribution of this work, that offers specific defining equations for the function $B'$.

\subsubsection{Computations}\label{comp}
\noindent The following were performed in the computational software \texttt{Maple}.
\subsection*{Parameter: 3-cycles}
\noindent For dissections, the defining polynomial $p_3(D,z,u)$ is the following:

\[p_3=-u{D}^{3}+u{D}^{2}z-D{z}^{3}+{z}^{4}+{D}^{3}+{D}^{2}z-D{z}^{2}\]

\noindent Its discriminant with respect to $D$, $\mathrm{disc}(z,u)$, is equal to 
\[ -{z}^{6}  \left( 4\,{u}^{3}z+8\,{u}^{2}{z}^{2}+4\,u
{z}^{3}-8\,{u}^{2}z-44\,u{z}^{2}-4\,{z}^{3}-{u}^{2}+20\,uz+32\,{z}^{2}
+2\,u+8\,z-5 \right). \]
From this, we retrieve the root $r(1)=3-2\sqrt{2}$ . By setting $z=r(u)$ and differentiating with respect to $u$ in $\mathrm{disc}(r(u),u)$, we also retrieve $r'(1)=-\frac{3}{2}+\sqrt{2},$ $r''(1)=\frac{3\sqrt{2}}{4}-1$.

\noindent By differentiating $p_3$ with respect to $D$, we obtain exact expressions for the derivatives $\frac{\partial D(z,u)}{\partial u}$ and $\frac{\partial D(z,u)}{\partial z}$:
\[\frac{\partial D(z,u)}{\partial u}=-{\frac { \left( D \left( z,u \right)  \right) ^{2} \left( D \left( z,
u \right) -z \right) }{3\,u \left( D \left( z,u \right)  \right) ^{2}-
2\,D \left( z,u \right) uz+{z}^{3}-3\, \left( D \left( z,u \right) 
 \right) ^{2}-2\,D \left( z,u \right) z+{z}^{2}}}\]

\[\frac{\partial D(z,u)}{\partial z}={\frac {u \left( D \left( z,u \right)  \right) ^{2}-3\,D \left( z,u
 \right) {z}^{2}+4\,{z}^{3}+ \left( D \left( z,u \right)  \right) ^{2}
-2\,D \left( z,u \right) z}{3\,u \left( D \left( z,u \right)  \right) 
^{2}-2\,D \left( z,u \right) uz+{z}^{3}-3\, \left( D \left( z,u
 \right)  \right) ^{2}-2\,D \left( z,u \right) z+{z}^{2}}}\]
 \noindent Then, we write $1-zB''(z,1)=0$ in terms of $D$, using the previous expressions, i.e.
 \begin{equation}1+z \left( \,{\frac {D}{2{z}^{2}}}-\,{\frac {-3\,D{z}^{2}+4\,{z}^
{3}+2\,{D}^{2}-2\,Dz}{2z \left( {z}^{3}-4\,Dz+{z}^{2} \right) }}-\frac{1}{2}
 \right) =0\label{tau}\end{equation} and solve the system of Equation~(\ref{tau}) and $p_3(D,z,1)=0$. The values we obtain are $D \approx 0.04709517290,$ $ \tau  \approx 0.1707649868$. Then $\rho (1)=\tau {{\rm \exp}{(-\,{\frac {D \left( \tau,1 \right) }{2\tau }}-\frac{\tau}{2})}}\approx 0.1365937336$.

To compute the derivatives of $\Psi (\tau (u),u)$, we write them in terms of $D(\tau (u),u)$. The value $\tau '(1)$ can be found similarly, after writing the equation $\tau (u)B''(\tau (u),u)=1$ in terms of $D(\tau (u),u)$.  
In particular, we obtain $\tau '(1)\approx -0.849388502$, $\rho '(1)\approx -0.5564505691$ and $\rho ''(1)\approx 0.3078771691$. The final values are computed as indicated in Theorem~\ref{quasi}.

\subsection*{Parameter: 4-cycles}

The procedure of the computations is the same as in the previous case. We only note that the defining polynomial $p_4(D,z,u)$ is the following:
\begin{eqnarray*}p_4 &=& {u}^{4}{z}^{2}{D}^{6}-2\,{u}^{4}z{D}^{7}+{u}^{4}{D}^
{8}+2{u}^{3}{z}^{6}{D}^{3}-4{u}^{3}{z}^{5}{D}^{4}+2{
u}^{3}{z}^{4}{D}^{5}+{u}^{2}{z}^{10}-2\,{u}^{2}{z}^{9}D +\\ & & +{u
}^{2}{z}^{8}{D}^{2}-2\,{u}^{3}{z}^{4}{D}^{4}+4\,{u}^{3}{z}
^{3}{D}^{5}-4\,{u}^{3}{z}^{2}{D}^{6}+6\,{u}^{3}z{D}^
{7}-4\,{u}^{3}{D}^{8}-2\,{u}^{2}{z}^{8}D+\\ & & +4\,{u}^{2}{z}^{7}
{D}^{2} -6\,{u}^{2}{z}^{6}{D}^{3}+10\,{u}^{2}{z}^{5}{D_{{1}
}}^{4}-6\,{u}^{2}{z}^{4}{D}^{5}-2\,u{z}^{10}+4\,u{z}^{9}D-
2\,u{z}^{8}{D}^{2}+{u}^{2}{z}^{6}{D}^{2}- \\ & & -2\,{u}^{2}{z}^{5}
{D}^{3}+3\,{u}^{2}{z}^{4}{D}^{4}-6\,{u}^{2}{z}^{3}{D
}^{5}+5\,{u}^{2}{z}^{2}{D}^{6}-6\,{u}^{2}z{D}^{7}6\,{u}^{
2}{D}^{8}+2\,u{z}^{8}D-4\,u{z}^{7}{D}^{2}+4\,u{z}^{6
}{D}^{3}-\\ & &-8\,u{z}^{5}{D}^{4}+6\,u{z}^{4}{D}^{5}{z}^{
10}-2\,{z}^{9}D+{z}^{8}{D}^{2} +u{z}^{5}{D}^{3}-2\,u{
z}^{4}{D}^{4}+3\,u{z}^{3}{D}^{5}-2\,u{z}^{2}{D}^{6}+
2\,uz{D}^{7}-\\ & & -4\,u{D}^{8}+{z}^{9}-2\,{z}^{8}D+{z}^{7}
{D}^{2}+2\,{z}^{5}{D}^{4}-2\,{z}^{4}{D}^{5} -{z}^{7}D
_{{1}}+2\,{z}^{6}{D}^{2}-{z}^{5}{D}^{3}+{z}^{4}{D}^{
4}-{z}^{3}{D}^{5}+{D}^{8}
\end{eqnarray*}
and gives the intermediate values $r'(1)\approx -15/2+21\sqrt{2}{4}$, $r'(1)\approx -413/4+2337\sqrt{2}{32}$, $\tau \approx 0.1707649868$,  $D(\tau ,1) \approx 0.4709517290$, $\tau '(1)\approx -0.7427876522$.

\subsection{Restricted classes}\label{restrict}

Now we apply the results of Theorem~\ref{main} in the context of asymptotic enumeration. We give various examples in restricted classes of polygon dissections and outerplanar graphs.

\subsubsection*{Avoiding 3 and 4-cycles}

\noindent Using the equations of the previous section, we obtain immediately sets of equations for polygon dissections avoiding $3$ and $4$-cycles.  Setting $u=0$ and substituting for $\bar{D}$, we obtain
\begin{eqnarray*}\bar{D} =\bar{D}^3+\bar{D}\bar{D}_{\circ  }-z\bar{D}+z\end{eqnarray*} for $3$-cycles and 
 \begin{eqnarray*} \bar{D} &=& \bar{D}_{\circ }+\bar{D}_{3[1]}\\ \bar{D}_{\circ } &=& \bar{D}^4+\bar{D}\bar{D}_{\circ }-z\bar{D}+z, \\ \bar{D}_{3[1]} &=& \bar{D}_{\circ }^2 \end{eqnarray*} for $4$-cycles.

\subsubsection*{Avoiding 5-cycles}
\noindent In Figure~\ref{fig:2}, there are all the composite roots when $5$-cycles are avoided. We obtain the following system, after setting $u=0$ where appropriate.

\begin{minipage}{.5\textwidth }
\begin{eqnarray*} \bar{D} &=& \bar{D}_{\circ }+\bar{D}_{3}+\bar{D}_{4} \\  
\bar{D}_{\circ } &=&  \bar{D}^5+\bar{D}\bar{D}_{\circ }-z\bar{D}+z\\   
\bar{D}_{3[1]} &=& \bar{D}_{\circ }^2\\ 
\bar{D}_{3[2,3]} &=& \bar{D}_{\circ }^3\\ 
\end{eqnarray*} 
\end{minipage}
\begin{minipage}{.5\textwidth }
\begin{eqnarray*} 
\bar{D}_{4[1]} &=& \bar{D}_{\circ }^3   \\
\bar{D}_{4[2,3,4]} &=& \bar{D}_{\circ }^2(  \bar{D}_{\circ }+ \bar{D}_{4})^3   \\
\bar{D}_{4[5,6,7]} &=&   \bar{D}_{\circ } (\bar{D}_{\circ } +\bar{D}_{4})^6   \\
\bar{D}_{4[8]} &=&  (\bar{D}_{\circ } +\bar{D}_{4})^9    \\
\end{eqnarray*}
\end{minipage}

\noindent The system can be simplified, after observing that $\bar{D}_4=(\bar{D}_{\circ } +\bar{D}_{4})^3$ and $\bar{D}_3=\bar{D}_{\circ }^2(1+2\bar{D}_{\circ })$.

\begin{figure}[h!]\centering
  \includegraphics[scale=7]{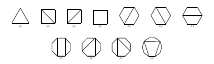}
  \caption{The composite roots when $5$-cycles are avoided. The dissection $4[8]$ is the only maximal one among them.}\label{fig:2}
\end{figure}

\subsubsection*{Avoiding 6-cycles}

\noindent In Figure~\ref{fig:4}, there are the composite roots when $6$-cycles are excluded, except for the ones including a $5$-gon. For $\bar{D}_5$, we observe immediately that $\bar{D}_5=(\bar{D}_{\circ }+\bar{D}_{4}+\bar{D}_{5})^4$.

\begin{minipage}{.5\textwidth }
\begin{eqnarray*} \bar{D} &=& \bar{D}_{\circ }+\bar{D}_{3}+\bar{D}_{4}+\bar{D}_{5} \\  
\bar{D}_{\circ } &=&  \bar{D}^6+\bar{D}\bar{D}_{\circ }-z\bar{D}+z\\   
\bar{D}_{3[1]} &=& \bar{D}_{\circ }^2\\ 
\bar{D}_{3[2,3]} &=& \bar{D}_{\circ }^3\\ 
\bar{D}_{3[4-8]} &=& \bar{D}_{\circ }^4\\ 
\bar{D}_{3[9,10]} &=& \bar{D}_{\circ }(\bar{D}_{\circ }+\bar{D}_5)^3\\ 
\bar{D}_{3[11]} &=& (\bar{D}_{\circ }+\bar{D}_5)^6\\ 
\end{eqnarray*} 
\end{minipage}
\begin{minipage}{.5\textwidth }
\begin{eqnarray*} 
\bar{D}_{4[1]} &=& (\bar{D}_{\circ }+\bar{D}_5)^3   \\[1.5pt]
\bar{D}_{4[2,3,4]} &=& \bar{D}_{\circ }^2(\bar{D}_{\circ }+\bar{D}_5)^2    \\[1.5pt]
\bar{D}_{4[5-10]} &=&  \bar{D}_{\circ }(\bar{D}_{\circ }+\bar{D}_5)^5   \\[1.5pt]
\bar{D}_{4[11,12,13]} &=&  (\bar{D}_{\circ } +\bar{D}_{5})^8    \\[1.5pt]
\bar{D}_{5} &=&  (\bar{D}_{\circ }+\bar{D}_{4}+\bar{D}_{5})^4   \\[1.5pt]
\end{eqnarray*}
\end{minipage}

\noindent We can immediately group all the $\bar{D}_{3[i]}$ and $\bar{D}_{4[j]}$ together to form equations for $\bar{D}_3$ and $\bar{D}_4$, respectively.

\subsubsection{Avoiding other patterns}

Now we avoid non-cyclic patterns. We analyse the ones induced by the dissections $3[2]$ and $4[2]$ in Figure~\ref{fig:4}, separately and together. We refer to them as \textit{Pattern I} and \textit{Pattern II}, respectively.

For Pattern I, the composite roots are the dissections $3[1,9,10,11] $ and $4[1]$ in Figure~\ref{fig:4} and the equations are the following: 

\begin{minipage}{.5\textwidth }
\begin{eqnarray*}\bar{D} &=& \bar{D}_{\circ }+\bar{D}_{3}+\bar{D}_{4[1]}\\ \bar{D}_{\circ } &=& \bar{D}^4+\bar{D}\bar{D}_{\circ }-z\bar{D}+z\\  \bar{D}_{4[1]} &=& (\bar{D}_{\circ }+\bar{D}_{3}+\bar{D}_{4[1]})^3\\ \end{eqnarray*} 
\end{minipage}
\begin{minipage}{.5\textwidth }
\begin{eqnarray*} \bar{D}_{3} &=& \bar{D}_{\circ }^2\\
\bar{D}_{3[9,10]} &=&( \bar{D}_{\circ }+ \bar{D}_{3}+ \bar{D}_{4[1]})^3 \bar{D}_{\circ }\\
\bar{D}_{3[11]} &=& ( \bar{D}_{\circ }+ \bar{D}_{3}+ \bar{D}_{4[1]})^6\\
 \end{eqnarray*}
\end{minipage}

 For Pattern II, the composite roots are the ones in Figure~\ref{fig:2} and the ones containing some $5$-gon. The equations are the same as in the $5$-cycle case, with the following differences: Now $\bar{D}= \bar{D}_{\circ }+\bar{D}_{3}+\bar{D}_{4}+\bar{D}_5$. Also, in all the equations apart from the one of $\bar{D}$ and $\bar{D}_{\circ }$, we substitute $\bar{D}_{\circ }$ for $\bar{D}_{\circ }+\bar{D}_5$. The equation $\bar{D}_5=\bar{D}^4$ must be added as well.

\begin{figure}[h!]\centering
  \includegraphics[scale=7]{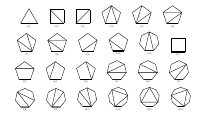}

  \caption{The composite roots when $6$-cycles are excluded, except for the ones including a $5$-gon.}\label{fig:4}
\end{figure}

For Patterns I and II, the composite roots are the dissections $3[1],4[1-8]$ in Figure~\ref{fig:2} and the ones containing some $5$-gon. The equations are the same as when avoiding Pattern II, only now $\bar{D}_{3[2]}$ and $\bar{D}_{3[3]}$ are omitted.

\begin{theorem}\label{enumm}
Let $\mathcal{D}, \mathcal{G}$ be the classes of dissections and outerplanar graphs avoiding a set of 2-connected patterns $\Delta =\{\delta _1 ,...,\delta _m\}$, respectively. Then, $\mathcal{D}$ has asymptotic growth of the form \[\alpha _n\sim \frac{\alpha}{\Gamma (-\frac{1}{2})}\cdot n^{-3/2} \cdot r ^{-n} \] and $\mathcal{G}$ has asymptotic growth of the form \[g_n\sim \frac{g}{\Gamma (-\frac{3}{2})} \cdot n^{-5/2} \cdot\rho ^{-n}\cdot n!\] where 
both $\alpha ,g$ are computable constants. In Table~\ref{table:3}, there are approximations of $\alpha , g$ for various choices of $\Delta $. \end{theorem}
\begin{proof}
We apply Theorem~\ref{drmota} with the same reasoning as in Theorem~\ref{random} and, in the end, obtain singular expansions with singular exponents $1/2$ and $3/2$ for $\mathcal{D}$ and $\mathcal{G}$, respectively. Then, the types of asymptotic growth can be obtained from the transfer principles of singularity analysis. 

It is true that $g=\tau (\log \rho -\log (\tau )+1)+B(\tau )$ \cite{bodirsky2007enumeration}. The value $B(\tau )$ can be approximated from the systems in this work. For details on the computations, see Section~\ref{computations}.
\end{proof}

\begin{table}[h!]\centering\ra{1.5}
\begin{tabular}{l|lll|lll}
Restriction & $r $ & $r^{-1}$ & $\alpha $ &  $\rho $ & $\rho ^{-1}$ & $g$ \\
\midrule[.8pt]

3-cycles    &  0.29336 & 3.40869  & 0.02330 &  0.20836 & 4.79916 & 0.01578\\ 

4-cycles    &  0.26488 & 3.77515 & 0.02177 &  0.18919  & 5.28562 & 0.01462 \\ 

5-cycles     &  0.25383 & 3.93949 & 0.02217 & 0.18045 & 5.54143 & 0.01514 \\ 

6-cycles      &  0.24835 & 4.02657 & 0.02321 & 0.17510  & 5.71082 & 0.01630   \\                                                                                          \hline
pattern I     &  0.20867 & 4.79214  & 0.01592 & 0.15895  & 6.29100  &   0.01050      \\                                                                                              
pattern II     &  0.22416 &  4.46098 & 0.01856  & 0.16608  &  6.02092 & 0.01195    \\                                                                                              
pattern I\&II   &  0.24332 &  4.10977 & 0.01987  & 0.17751  &  5.63345 & 0.01351    \\                                                                                              

\end{tabular}
\caption{The constants for the asymptotic growth of restricted polygon dissections and outerplanar graphs, respectively.}\label{table:3}
\end{table}

\subsubsection{Computations}
\label{computations}

We will use the notation from the proof of Theorem~\ref{random}.

\noindent  In the level of dissections, there is no computational difficulty, since in all cases the constants can be computed either through the defining equation of $\bar{D}$, or using the characteristic system. In all cases, the main singularity can be found by solving directly the characteristic system, while $\alpha $ can be found by substituting the $\bar{D}_i$ variables in the system by their singular expansions and solve the system with respect to the undetermined coefficients.

Moving to the connected and general level, some values are harder to compute. In particular, the computation of $B(\tau )$ is not easily accessible through our implicit function setting, while also $\tau $ is hard to compute when the value $H_{\Delta  }$ grows and the defining equation for $d$ becomes either too big or too hard to compute. For instance, the defining polynomial for the 5-cycle case has degree 45 with respect to $D$ and 54 with respect to $z$, while the polynomial for the 6-cycle case was not retrieved in a reasonable amount of time, i.e. in half hour. For our purposes, we computed an approximation for both values $\tau ,B (\tau ) $, using the first 700 terms of the power series expansion of $B$. The expansion was extracted from the one of $\bar{D}$, which was found by iterating the defining system in \texttt{Maple}. The results are displayed in Table~\ref{table:3}.

\section{Acknowledgements}
\noindent This research was funded under an FPI grant from the MINECO research project MTM2015-67304-PI. The author was also partially funded by the Barcelona Graduate School of Mathematics, funded by Maria de Maetzu research grant MDM-2014-0445. The author is grateful to Prof. Juanjo Rué for posing the problem and for making valuable comments on the draft. The author is also grateful to Prof. Dimitrios M. Thilikos who made this research possible. The anonymous referees are warmly thanked for their remarks that improved significantly the initial version of this work.

\section{Appendix}

\subsection{Counting series}

\noindent Here are the first terms of the counting sequences of all the restricted dissection classes that appear in Section~\ref{restrict}. The count for triangle-free polygon dissections appears also in \cite{smiley} and is the sequence \texttt{A046736} in~\cite{Sloane_theencyclopedia}. The counts for $3,4,5,6$-cycle free dissections appear also in the author's {master's thesis}~\cite{velonamsc}, where they are derived using ad hoc arguments.

\begin{table}[h]\label{count}
\centering\centering\ra{1.2}
\begin{tabular}{l|lllllll}

 n& 3-cycles & 4-cycles & 5-cycles & 6-cycles & pattern I  & pattern II & patterns I \& II  \\
 \midrule[.4pt]
2  &1 & 1 & 1 & 1 & 1 & 1 & 1\\ 
3 &0 & 1 & 1 & 1 & 1 & 1 & 1 \\ 
4 &1 & 0 & 3 & 3 & 1 & 3 & 1\\ 
5 &1 & 1 & 0 & 11 & 6 & 1 & 1\\ 
6 &4 & 7 & 4 & 0 & 19 & 10 & 10\\ 
7 &8 & 22 & 8 & 15 & 64 & 43 & 29\\ 
8 &25 & 49 & 65 & 37 & 251 & 181 & 101 \\ 
9 &64 & 130 & 229 & 85 & 979 & 643 & 283\\ 
10 &191 & 468 & 946 & 651 & 3888 & 2233 & 1023\\
11&540&1651&2850&2498&15896&8152&3576
\\ 12&1616&5240&9367&10556&65871&31523&13143
\\ 13&4785&16485&28068&46112&276225&125776&46502
\\ 14&14512&55184&97408&167100&1171838&502449&169221
\\ 15&44084&190724&339694&621677&5016697&2001773&
618807\\ 16&135545&652359&1276467&2215039&21644451&
8002279&2301983\\ 17&418609&2213044&4659990&7524303&
94033342&32271594&8576756\\ 18&1302340&7584939&
17107629&26414280&410990601&131355333&32169753\\ 19&
4070124&26346522&61200635&92579458&1805881012&538125069&121134235
\\ 20&12785859&91951596&220323189&332018450&
7972740040&2213876868&458881370\\

\bottomrule[.4pt]

\end{tabular}
\caption{The first terms of the restricted dissection classes of Section~\ref{app}.}
\end{table}

\end{document}